\documentclass{amsart}

\usepackage{amssymb}
\usepackage{amsmath}
\usepackage{amsthm}

\newtheorem {theorem}{Theorem}[section]
\newtheorem {prop}[theorem]{Proposition}

\newtheorem {lemma}[theorem]{Lemma}

\theoremstyle{definition}
\newtheorem{df}{Definition}[section]
\numberwithin{equation}{section}

\newcommand{\FF}{\Phi}

\newcommand{\grwhar}{h^{\we}}
\newcommand{\harm}{\mathcal{H}}

\newcommand{\sid}{\sigma_d}

\newcommand{\za}{\zeta}
\newcommand{\we}{w}

\newcommand{\PH}{\Phi}

\newcommand{\har}{h}
\newcommand{\hol}{H}

\newcommand{\spd}{\partial B_d}
\newcommand{\cn}{{\mathbb{C}}^m}
\newcommand{\rd}{{\mathbb{R}}^d}
\newcommand{\bd}{B_d}
\newcommand{\bbn}{{\mathbb B}_m}
\newcommand{\disk}{{\mathbb B}_1}
\newcommand{\Rbb}{\mathbb R}

\newcommand{\Nbb}{\mathbb N}
\newcommand{\Zbb}{\mathbb Z}

\begin{document}

\title{Harmonic approximation by finite sums of moduli}



\author{Evgueni Doubtsov}


\address{St.~Petersburg Department
of V.A.~Steklov Mathematical Institute, Fontanka 27,
St.~Petersburg 191023, Russia}

\email{dubtsov@pdmi.ras.ru}

\begin{abstract}
Let $h(B_d)$ denote the space of real-valued harmonic functions
on the unit ball $B_d$ of $\mathbb{R}^d$, $d\ge 2$.
Given a radial weight $w$ on $B_d$, consider the following problem:
construct a finite family $\{f_1, f_2, \dots, f_J\}$ in $h(B_d)$
such that the sum $|f_1| + |f_2|+\dots + |f_J|$ is equivalent to $w$.
We solve the problem for weights $w$ with a doubling property.
Moreover, if $d$ is even, then we characterize those
 $w$ for which the problem has a solution.
\end{abstract}

\thanks{The author was supported by the Russian Science Foundation (grant No. 14-41-00010).}

\maketitle

\section{Introduction}\label{s_int}
\subsection{Weight functions and radial weights}\label{ss_wfun_df}
By definition, $\we: [0, 1)\to (0, +\infty)$
is a weight function if $\we$
is a non-decreasing, continuous, unbounded function.
Let $B$ denote the unit ball of a real or complex Euclidean space.
We extend $\we$ to a radial weight on $B$ setting
$\we(z) = \we(|z|)$, $z\in B$.

For functions $u, v: B\to (0,+\infty)$, we write $u\asymp v$
and we say that $u$ and $v$ are equivalent if
\[
C_1 u(z) \le v(z) \le C_2 u(z), \quad z\in B,
\]
for some constants $C_1, C_2 >0$.
The definition of equivalent functions on $[0,1)$ is analogous.

\subsection{Main result}\label{ss_int_main}
Let $\har(\bd)$ denote
the space of real-valued harmonic functions on the unit ball $\bd$
of $\rd$, $d\ge 2$.
Harmonic approximation mentioned in the title of the present paper
refers to the following notion:

\begin{df}
A weight function $\we: [0, 1) \to (0, +\infty)$ is called \textit{harmonically approximable}
if there exists a finite family $\{f_1, f_2, \dots, f_J\}\subset\har(\bd)$
such that
\begin{equation}\label{e_apprx}
|f_1|+|f_2|+\dots + |f_J| \asymp \we,
\end{equation}
where $\we$ is extended to a radial weight on $\bd$.
Clearly, property~\eqref{e_apprx} does not change if $\we$
is replaced by an equivalent weight function.
%$\we_1$ such that $\we_1\asymp\we$.
\end{df}

To study harmonic approximation, we use doubling and log-convex weight functions.
A weight function $\we$
is called \textit{doubling} if there exists
a constant $A >1$ such that
\begin{equation}\label{e_reg_df_we}
\we(1 - s/2) \le A \we(1 - s),
\quad 0<s\le 1.
\end{equation}
A weight function $\we$
is called \textit{log-convex} if
$\log \we(r)$ is a convex function of $\log r$, $0<r<1$.
The doubling condition~\eqref{e_reg_df_we} restricts the growth of $\we$;
a log-convex weight function may grow arbitrarily rapidly.
Also, it is known that every doubling weight function
is equivalent to a log-convex one.

\begin{theorem}\label{t_main}
Let $d\ge 2$ and let $\we$ be a weight function on $[0,1)$.
\begin{itemize}
\item[(i)]
If $\we$ is harmonically approximable, then $\we$ is equivalent to a log-convex weight function.
\item[(ii)]
If $d$ is even, then $\we$ is harmonically approximable
if and only if  $\we$ is equivalent to a log-convex weight function.
\item[(iii)]
If $d$ is odd and $\we$ is doubling, then $\we$ is harmonically approximable.
\end{itemize}
\end{theorem}

\subsection{Comments}\label{ss_comm}
\subsubsection{Holomorphically approximable weights}
Let $\hol(\bbn)$ denote
the space of holomorphic functions on the unit ball $\bbn$
of $\cn$, $m\ge 1$.

\begin{df}
A weight function $\we$ is called \textit{holomorphically approximable}
if there exists a finite family $\{f_1, f_2, \dots, f_J\}\subset\hol(\bbn)$
such that property~\eqref{e_apprx} holds in $\bbn$.
\end{df}

A characterization of holomorphically approximable
weight functions is given by the following result.

\begin{theorem}[{\cite[Theorems~1.2 and 1.3]{AD13}}]\label{t_ad_logconv}
Let $m\ge 1$. A weight function $\we$
is holomorphically approximable if and only if $\we$
is equivalent to a log-convex weight function.
\end{theorem}

In particular, Theorem~\ref{t_main}(ii) follows from
Theorem~\ref{t_main}(i) and Theorem~\ref{t_ad_logconv}.
Indeed, given a log-convex weight function extended to a radial weight on $B_{2m}$,
it suffices to identify $B_{2m}$ and $\bbn$,
and to take the real and imaginary parts of holomorphic
functions provided by Theorem~\ref{t_ad_logconv}.

To the best of the author's knowledge,
Theorem~\ref{t_main}(iii) is new for any weight function $\we$.

\subsubsection{Optimal $J$ in \eqref{e_apprx}}
We always have $J\ge 2$ in \eqref{e_apprx}
by the mean value property for harmonic functions.
The corresponding optimal number $J=J(d)$
is probably of independent interest.

\subsubsection{$L^2$-means}
For $f\in\har(\bd)$, put
\[
M_2^2(f, r) = \int_{\spd} |f(r y)|^2\, d\sid(y),
\quad 0\le r<1,
\]
where $\sid$ is the normalized Lebesgue measure on the
sphere $\spd$.
The proof of Theorem~\ref{t_main}(i)
uses the following modification of
property~\eqref{e_apprx} in terms of $L^2$-means:
\begin{equation}\label{e_har_M2}
\begin{aligned}
\textrm{there exists\ } J\in\Nbb \textrm{\ and\ }
f_1,\dots, f_J\in\har(\bd)
\textrm{\ such that\ }
\\
M_2^2(f_1, r) + \dots + M_2^2(f_J, r) \asymp \we^2(r),\quad 0\le r<1.
\end{aligned}
\end{equation}
If the above property holds, then $\we$ is called harmonically $L^2$-approximable.
We use an analogous definition for $\hol(\bbn)$.

As mentioned above, we have $J\ge 2$ in \eqref{e_apprx}
for any weight function $\we$.
However, if \eqref{e_apprx} is replaced by \eqref{e_har_M2},
then we need just one function $f\in\har(\bd)$;
see Proposition~\ref{p_L2} below.
The situation is similar for $\hol(\bbn)$:
one has $J\ge 2$ in the holomorphic analog of \eqref{e_apprx} by the maximum principle, but
the approximation problem for $L^2$-means is solvable
by one function $f\in\hol(\bbn)$.

\begin{prop}\label{p_L2}
Let $\we$ be a weight function on $[0,1)$.
Then the following properties are equivalent:
\begin{align}
&\text{$\we$ is equivalent to a log-convex function;}\label{e_L2_logconv}
\\
&\text{$\we$ is harmonically $L^2$-approximable;}\label{e_L2_har}
\\
&\text{there exists $f\in\har(\bd)$
such that $M_2(f, r) \asymp \we(r)$, \quad $0\le r <1$;}\label{e_L2_har_1}
\\
&\text{$\we$ is holomorphically $L^2$-approximable;}\label{e_L2_hol}
\\
&\text{there exists $f\in\hol(\bbn)$
such that $M_2(f, r) \asymp \we(r)$, \quad $0\le r <1$.}\label{e_L2_hol_1}
\end{align}
\end{prop}

\subsubsection{Growth spaces}
One may consider property~\eqref{e_apprx}
as a reverse estimate in the growth space $\grwhar(\bd)$, $d\ge 2$.
By definition, $\grwhar(\bd)$ consists of those $f\in\har(\bd)$
for which
\begin{equation}\label{e_grwhar_def}
|f(x)|\le C \we(x), \quad x\in\bd.
\end{equation}
Property~\eqref{e_apprx}
guarantees that estimate~\eqref{e_grwhar_def} is, in a sense, reversible
for a finite family of test functions in $\grwhar(\bd)$.
Such test functions are known to be useful in the studies of various concrete operators
on the growth spaces
(see, for example, \cite{AD12, KP11} and references therein).

\subsection{Organization of the paper}\label{ss_org}
The results related to the log-convexity are collected
in Section~\ref{s_logconv}: we prove Theorem~\ref{t_main}(i)
and Proposition~\ref{p_L2}.
Theorem~\ref{t_main}(iii), the main technical result of the present paper,
is obtained in Section~\ref{s_doub}.

\section{Log-convex weight functions}\label{s_logconv}

\subsection{Equivalence to a log-convex weight function is necessary}\label{ss_necessary}

The following lemma is standard (see, for example, \cite{KM94}).
\begin{lemma}\label{l_KM}
Let $\we$ be a weight function.
If $\we$ is harmonically %(or holomorphically)
$L^2$-approximable,
then $\we$ is equivalent to a log-convex weight function.
\end{lemma}
\begin{proof}
Assume that
\[
\we^2(r) \asymp
M_2^2(f_1, r) +\dots + M_2^2(f_J, r), \quad 0\le r <1,
\]
for a family $\{f_1,\dots, f_J\}\subset \har(\bd)$.
Put $M_2^2(r)= M_2^2(f_1, r) +\dots + M_2^2(f_J, r)$.
We claim that $M_2(r)$ is log-convex.
Indeed, for $j=1,\dots, J$, we have
\[
f_j(x) = \sum_{k=0}^\infty P_{j,k}(x),\quad x\in\bd,
\]
where $P_{j,k}$ is a harmonic homogeneous polynomial of degree $k$,
and the series converges uniformly on compact subsets of $\bd$.
For $k_1\neq k_2$, $P_{k_1}$ and $P_{k_2}$
are orthogonal in $L^2(\spd)$, hence,
\[
M_2^2(f_j, r) = \sum_{k=0}^\infty \|P_{j,k}\|^2_{L^2(\spd)} r^{2k},\quad 0\le r<1.
\]
So, $M_2^2(r) = \sum_{k=0}^\infty a_k^2 r^{2k}$ for certain $a_k\in\Rbb$.
Therefore, Hadamard's three circles theorem or direct computations
guarantee that $M_2^2(r)$ and $M_2(r)$ are log-convex,
as required.
\end{proof}

Applying Lemma~\ref{l_KM}, we obtain Theorem~\ref{t_main}(i)
and several implications in Proposition~\ref{p_L2}.

\begin{proof}[Proof of Theorem~\ref{t_main}(i)]
We are given a family
$\{f_1,\dots, f_J\}\subset \har(\bd)$ such that
\[
\we(x) \asymp
|f_1 (x)| +\dots + |f_J(x)|, \quad x\in\bd,
\]
or, equivalently,
\[
\we^2(r) \asymp
|f_1 (r y)|^2 +\dots + |f_J(r y)|^2, \quad 0\le r <1,\ y\in\spd.
\]
Integrating over the sphere $\spd$ with respect to Lebesgue measure $\sid$,
we obtain
\[
\we^2(r) \asymp
M_2^2(f_1, r) +\dots + M_2^2(f_J, r), \quad 0\le r <1.
\]
So, by Lemma~\ref{l_KM}, $\we$
is equivalent to a log-convex weight function.
\end{proof}

Clearly, the above argument also guarantees that every holomorphically approximable
weight function is equivalent to a log-convex one; see \cite{AD13} for a different proof.

\begin{proof}[Proof of Proposition~\ref{p_L2}]
The implications \eqref{e_L2_har_1}$\Rightarrow$\eqref{e_L2_har}
and \eqref{e_L2_hol_1}$\Rightarrow$\eqref{e_L2_hol} are trivial.
By Lemma~\ref{l_KM}, \eqref{e_L2_har} implies \eqref{e_L2_logconv},
and \eqref{e_L2_hol} implies \eqref{e_L2_logconv}.
So, to finish the proof of Proposition~\ref{p_L2},
it suffices to show that
\eqref{e_L2_logconv} implies \eqref{e_L2_har_1} and \eqref{e_L2_hol_1}.
\end{proof}

\subsection{Approximation by integral means}\label{ss_L2}
In this section, we show that
\eqref{e_L2_logconv} implies \eqref{e_L2_har_1}.
The proof of the implication  \eqref{e_L2_logconv}$\Rightarrow$\eqref{e_L2_hol_1}
is analogous.

\begin{lemma}\label{l_L2approx}
Let $\we$ be a log-convex weight function on $[0,1)$.
Then there exists a sequence $\{a_k\}_{k=0}^\infty \subset\Rbb$
such that
\begin{equation}\label{e_L2approx}
\sum_{k=0}^\infty a_k^2 r^{2k} \asymp \we^2(r), \quad 0\le r <1,
\end{equation}
where the series converges uniformly on compact subsets of $[0,1)$.
\end{lemma}
\begin{proof}
By \cite[Theorem~1.2]{AD13}, there exist $f_1, f_2\in \hol(\disk)$
such that $|f_1(z)| + |f_2(z)| \asymp \we(|z|)$, $z\in \disk$, hence,
\[
|f_1(r\za)|^2 + |f_2(r\za)|^2 \asymp \we^2(r), \quad 0\le r <1,\ \za\in\partial\disk.
\]
Integrating the above equivalence with respect to Lebesgue measure on the unit circle $\partial\disk$,
we obtain \eqref{e_L2approx} with $a_k^2 = |\widehat{f_1}(k)|^2 + |\widehat{f_2}(k)|^2$.
\end{proof}

\begin{proof}[Proof of Proposition~\ref{p_L2}: \eqref{e_L2_logconv}$\Rightarrow$\eqref{e_L2_har_1}]
Let $\harm_k = \harm_k(d)$
denote the space of harmonic homogeneous polynomials
of degree $k$ in $d$ real variables.
The same symbol is used for the restriction of $\harm_k$ to the sphere $\spd$.
Let $Z_k(\cdot, \cdot)$ denote the reproducing kernel for $\harm_k \subset L^2(\spd)$.
Fix a point $x\in\spd$. So, we have $Z_k(\cdot) = Z_k(x, \cdot) \in \harm_k$.
Set $Y_k = (\textrm{dim}\harm_k)^{-\frac{1}{2}} Z_k$. Then
\begin{align}
\|Y_k\|_{L^\infty(\spd)} &= \sqrt{\textrm{dim}\harm_k};\label{e_har_Li}
\\
\|Y_k\|_{L^2(\spd)} &= 1.\label{e_har_L2}
\end{align}
See, for example, \cite{DX13} for the above properties of zonal harmonics $Z_k$.

Given a log-convex weight function $\we$, let
the sequence $\{a_k\}_{k=0}^\infty$ be that provided by Lemma~\ref{l_L2approx}.
Put
\begin{equation}\label{e_zonal}
f(x) = \sum_{k=0}^\infty a_k Y_{k}(x),\quad x\in \bd.
\end{equation}
Property~\eqref{e_har_Li} and explicit formulas for $\textrm{dim}\harm_k$
guarantee that the above series converges uniformly on compact subsets of $\bd$.
Hence, $f\in\har(\bd)$.
Since \eqref{e_zonal} is an orthonormal series, we have
\[
M_2^2(f,r) = \sum_{k=0}^\infty a_k^2 r^{2k} \asymp \we^2(r),\quad 0\le r<1,
\]
by Lemma~\ref{l_L2approx}.
So, the proof of the proposition is finished.
\end{proof}

\section{Doubling weight functions}\label{s_doub}

In the present section, we prove Theorem~\ref{t_main}(iii).
Both in \cite{AD12} and \cite{AD13}, the required holomorphic functions
in the complex ball $\bbn$, $m\ge 2$, are constructed
as appropriate lacunary series of Aleksandrov--Ryll--Wojtaszczyk polynomials.
As far as the author is concerned, the existence
of analogous homogeneous harmonic polynomials in $\mathbb{R}^{2n+1}$, $n=1,2,\dots$,
remains an open problem.
So, in the proof of Theorem~\ref{t_main}(iii), we use series of special harmonic functions
that are not polynomials.

\subsection{Building blocks}\label{s_build}
\begin{lemma}[see {\cite[Lemma~3]{EM12}}]\label{l_em}
Let $d\ge 2$ and let $p\in\Nbb$.
Then there exist constants $\alpha=\alpha(d)\in\Nbb$, $Q=Q(d)\in\Nbb$ and $C=C(p,d)>0$
with the following property:
for every $n\in\Zbb_+$, there exist functions
$u_{q, n}\in \har(\bd)$, $q=1, \dots, Q$,
such that
\begin{align}
|u_{q, n}(x)| &\le 1, \quad x\in B_d; \label{e_em1}
\\
\max_{1\le q\le Q}|u_{q, n}(x)| &\ge \frac{1}{4},\quad
0< 1-|x| < 2^{-\alpha-n}; \label{e_em2}
\\
|u_{q, n}(x)| &\le C(p, d) 2^{-np} (1-|x|)^{-p}, \quad x\in B_d. \label{e_em3}
\end{align}
\end{lemma}
\begin{proof}
Let $n\in\Zbb_+$. By \cite[Lemma~3]{EM12},
there exists a function $v_1\in\har(\bd)$ with properties \eqref{e_em1}, \eqref{e_em3},
and such that
\[
|v_1(ry)| \ge \frac{1}{4},\quad 0< 1-r < 2^{-\alpha-n}\text{\ and\ } y\in E_n,
\]
where
\begin{align*}
E_n &=\left\{y\in\spd:\ y=(t\cos\varphi, t\sin\varphi, y_3, \dots, y_d),\ t\ge\frac{3}{4},\ \varphi\in\Phi_n \right\},
\\
\Phi_n &=\bigcup_{j=0}^{2^n-1}
\left[\left((2j+\frac{1}{4}\right)2^{-n}\pi,\, \left(2j+\frac{3}{4}\right)2^{-n}\pi \right].
\end{align*}
Rotating the set $E_n$, we obtain $v_1, v_2, v_3, v_4\in \har(\bd)$ such that
\[
|v_1(ry)| +\dots + |v_4(ry)| \ge \frac{1}{4},\quad 0< 1-r < 2^{-\alpha-n}\text{\ and\ } y\in E_0,
\]
where $E_0 =\{y\in\spd:\ y_1^2 +y_2^2 \ge \frac{9}{16}\}$.
Now, observe that there exists a finite family of rotations $\{T_k\}_{k=1}^K$
such that
\[
\bigcup_{k=1}^K T_k E_0 =\spd.
\]
So, the family
$v_m T_k^{-1}\in \har(\bd)$, $m=1,2,3,4$, $k=1, \dots, K$, has the required properties
for the number $n\in\Zbb_+$ under consideration.
\end{proof}

\subsection{Construction}
We are given a doubling weight function $\we: [0,1)\to (0,+\infty)$.
Without loss of generality, assume that $\we(0)=1$.
We use the auxiliary function
\[
\PH(x) = \we\left( 1-\frac{1}{x}\right), \quad x\ge 1.
\]
So, we have $\PH(1)=1$ and $\we(t) = \PH\left( \frac{1}{1-t}\right)$, $0\le t<1$.
The doubling condition (\ref{e_reg_df_we}) rewrites as
\begin{equation}\label{e_reg_df}
\PH(2x) \le A \PH(x),
\quad x\ge 1.
\end{equation}
Without loss of generality, we assume that $A\ge 2$.

For $k=0, 1, \dots$, set
\begin{equation}\label{e_nk_df}
n_k = \max\{j\in\Zbb_+:\ \PH(2^j)\le A^{k} \}.
\end{equation}
Since $\PH(1)=1$, $n_0$ is correctly defined.
Also, we have $\PH(2\cdot 2^{n_k}) \le A \PH(2^{n_k}) \le A^{k+1}$
by \eqref{e_reg_df} and \eqref{e_nk_df}; hence, $n_{k+1}> n_k$
and $n_{k+\ell} - n_k \ge \ell$ for $\ell\in\Nbb$.
In what follows, we often use these properties without explicit reference.

Let the functions $u_{q,n}$ be those provided by Lemma~\ref{l_em}.
For $J\in\Nbb$, consider the series
\[
F_{q, j (x)} = \sum_{k=0}^\infty A^{Jk +j} u_{q, n_{Jk +j}}(x), \quad x\in\bd,\ q=1,2,\dots, Q,\ j=0,1,\dots, J-1.
\]
Estimates obtained below guarantee that
the above series uniformly converges on compact subsets of $\bd$; in particular,
$F_{q,j}$ is a harmonic function. The exact value of the constant $J$ will be selected later.

\subsection{Basic estimates}
Observe that it suffices to obtain the two-sided estimate
\begin{equation}\label{e_main_estimate}
\sum_{q=1}^Q \sum_{j=0}^{J-1} |F_{q,j}(x)| \asymp \PH \left(\frac{1}{1-|x|} \right)
\end{equation}
for $2^{-\alpha- n_{Jm+j+1}} \le 1-|x| \le 2^{-\alpha- n_{Jm+j}}$, $j=0,1,\dots, J-1$, $m=0,1,\dots$.

Indeed, \eqref{e_main_estimate} guarantees that
\[
1+ \sum_{q=1}^Q \sum_{j=0}^{J-1} |F_{q, j}(x)| \asymp \PH\left(\frac{1}{1-|x|} \right)
\]
for $1>|x|\ge 1- 2^{-\alpha -n_0}$, hence, for all $x\in\bd$.

\subsection{Lower estimate in \eqref{e_main_estimate}}\label{ss_low}
In fact, we are going to prove the following somewhat stronger property:
\begin{equation}\label{e_jlow}
\sum_{q=1}^Q |F_{j,q}(x)| \ge C \FF\left( \frac{1}{1-|x|}\right), \quad j=0,1,\dots, J-1,
\end{equation}
for a universal constant $C>0$
for $x\in\bd$ such that $2^{-\alpha- n_{Jm+j+1}} \le 1-|x| \le 2^{-\alpha- n_{Jm+j}}$, $m=0,1,\dots$.
So, fix an $m\in\{0,1,\dots\}$ and assume, without loss of generality, that $j=0$.
Now, consider a point $x$.

Since $1-|x| \le 2^{-\alpha-n_{Jm}}$, property \eqref{e_em2} guarantees that
 $|u_{\mathfrak{q}, n_{Jm}}(x)|\ge \frac{1}{4}$
for some $\mathfrak{q}=\mathfrak{q}(x)\in\{1,2,\dots, Q\}$.
Fix such a $\mathfrak{q}$ and consider the series
\[
F_\mathfrak{q}(x) = F_{\mathfrak{q}, 0}(x) = \sum_{k=0}^\infty A^{Jk} u_{\mathfrak{q}, n_{Jk}}(x).
\]
For a sufficiently large $J$, we will show that
\[
|F_{\mathfrak{q}}(x)| \ge C \PH\left( \frac{1}{1-|x|}\right),
\]
where $C>0$ is a universal constant.

Represent the series $F_{\mathfrak{q}}(x)$ as the sum of the following three functions:
\begin{equation}\label{e_sum3}
\sum_{k=0}^{m-1} + A^{Jm} u_{\mathfrak{q}, n_{Jm}}(x) + \sum_{k=m+1}^\infty := f_1(x) + f_2(x) + f_3(x).
\end{equation}
First, by \eqref{e_em1},
\[
|f_1(x)| \le \sum_{k=0}^{m-1} A^{Jk} \le A^{J(m-1) +1}.
\]
Second, by the definition of $\mathfrak{q}$,
\[
|f_2(x)| \ge \frac{A^{Jm}}{4}.
\]
Finally, given $p\in\Nbb$, property~\eqref{e_em3} guarantees that
\begin{align*}
|f_3(x)|
&\le C(p,d) \sum_{k=m+1}^\infty A^{Jk} 2^{-p n_{Jk}} (1-|x|)^{-p}
\\
&\le C(p,d) \sum_{k=m+1}^\infty A^{Jk} 2^{-p (n_{Jk} - n_{Jm +1})} 2^{-p n_{Jm + 1}} 2^{p(\alpha+ n_{Jm +1})}.
\end{align*}
Now, fix $p\in\Nbb$ such that $A < 2\cdot 2^p$. In particular, the constants $C(p,d)$ and $2^{p\alpha}$ are also fixed.
We claim that the series under consideration converges. Indeed, we have $n_{Jk} - n_{Jm +1} \ge J(k-m) -1$, thus,
\begin{align*}
|f_3(x)|
&\le  A^{Jm} C(p,d) 2^{p\alpha} \sum_{k=m+1}^\infty A^{J(k-m)} 2^{-p (n_{Jk} - n_{Jm +1})}
\\
&\le A^{Jm} C(p,d) 2^{p(\alpha+1)} \sum_{s=1}^\infty \left(\frac{A}{2^p} \right)^{Js}
\\
&\le A^{Jm} C(p,d) 2^{p(\alpha+1)} \sum_{s=1}^\infty \left(\frac{1}{2} \right)^{Js}.
\end{align*}
Since $p$ is fixed, we may select so large $J$ that
\[
C(p,d) 2^{p(\alpha+1)} \sum_{s=1}^\infty \left(\frac{1}{2} \right)^{Js} < \frac{1}{16}.
\]
Therefore, we have
\[
|f_3(x)| \le \frac{A^{Jm}}{16}.
\]

In sum,
we obtain
\begin{equation}\label{e_sum}
|F_{\mathfrak{q}}(x)| \ge \frac{A^{Jm}}{4} - A^{J(m-1) +1} - \frac{A^{Jm}}{16} \ge \frac{A^{Jm}}{8}
\end{equation}
for a sufficiently large $J$. Fix such a $J\in\Nbb$.
Note that the choice of appropriate $J$ does not depend on $m$.

Now, observe that $\FF$ is an increasing function,
hence, by \eqref{e_reg_df} and \eqref{e_nk_df},
\begin{equation}\label{e_FF}
\FF\left(\frac{1}{1-|x|} \right)\le \FF \left(2^{\alpha + n_{Jm +1}}\right)
\le A^\alpha\FF \left(2^{n_{Jm +1}}\right) \le A^{\alpha+1 +Jm}
\end{equation}
for $2^{-\alpha- n_{Jm+1}} \le 1-|x|$.

Finally, by \eqref{e_sum} and \eqref{e_FF}, we obtain
\[
\sum_{q=1}^Q |F_{q,0}(x)| \ge |F_{\mathfrak{q}}(x)| \ge \frac{A^{Jm}}{8}
\ge \frac{A^{-\alpha-1}}{8}\FF\left(\frac{1}{1-|x|} \right)
\]
for $2^{-\alpha- n_{Jm+1}} \le 1-|x| \le 2^{-\alpha- n_{Jm}}$.
In other words, the required lower estimate \eqref{e_jlow} holds for $j=0$ and $m=0,1,\dots$.
The parameter $J$ is fixed, so, the above
argument remains the same for arbitrary $j\in\{0,1,\dots, J\}$.
Therefore, the proof of the lower estimate in \eqref{e_main_estimate} is finished.

\subsection{Upper estimate in \eqref{e_main_estimate}}\label{ss_up}
We assume that the constants $J$ and $p$ are fixed according to the restrictions
of section~\ref{ss_low}; in particular, $A< 2\cdot 2^p$.

Given $m\in\{0,1,\dots\}$ and $j\in\{0,1, \dots, J-1\}$, suppose that
\[
2^{-\alpha - n_{Jm+j+1}} \le 1-|x| \le 2^{-\alpha - n_{Jm+j}}.
\]
Using \eqref{e_em1}, \eqref{e_em3} and the inequality $n_k- n_{Jm+j+1} \ge k-Jm-j-1$,
we obtain
\begin{equation}\label{e_QJ}
\begin{split}
\sum_{q=1}^Q
&\sum_{j=0}^{J-1} |F_{q, j}(x)|
\\
&\le
Q\left( \sum_{k=0}^{Jm+j} A^k + C(p,d)\sum_{k=Jm+j+1}^\infty A^k 2^{-p n_k} (1-|x|)^{-p}\right)
\\
&\le
Q\left( A^{Jm+j+1} + C(p,d) 2^{p\alpha}\sum_{k=Jm+j+1}^\infty A^k 2^{-p (n_k-n_{Jm+j+1})} \right)
\\
&\le Q\left( A^{Jm+j+1} + C(p,d) 2^{p\alpha}  A^{Jm+j+1} \sum_{s=0}^\infty \left(\frac{1}{2}\right)^s \right)
\\
&\le C(p, d, Q(d), \alpha(d)) A^{Jm+j+1}.
\end{split}
\end{equation}

Also, for $1-|x| \le 2^{-\alpha - n_{Jm+j}}$, we have
\begin{equation}\label{e_FFlow}
\FF\left( \frac{1}{1-|x|}\right) \ge \FF\left( 2^{\alpha+ n_{Jm+j}}\right)
\ge \FF\left( 2^{1+ n_{Jm+j}}\right) > A^{Jm+j}
\end{equation}
by the definition of $n_{Jm+j}$.

By \eqref{e_QJ} and \eqref{e_FFlow},
\[
\sum_{q=1}^Q \sum_{j=0}^{J-1} |F_{q, j}(x)| \le C(p, d, A) \PH\left(\frac{1}{1-|x|} \right)
\]
for all $x\in\bd$ under consideration.
In other words, the upper estimate in \eqref{e_main_estimate} holds.
The proof of Theorem~\ref{t_main}(iii) is finished.

\bibliographystyle{amsplain}

\end{document}